\documentclass[11pt, a4paper, leqno, oneside]{amsart}

\usepackage{pdflatex-maths-global-style}
\usepackage{local-style}

\bibliography{bibliography.bib}

\title{Decompositions of the stable module $\infty$-category}
\date{\today}

\author{Joshua Hunt}
\address{\parbox{\linewidth}{Department of Mathematical Sciences, University of Copenhagen, \newline Universitetsparken 5, 2100 København Ø, Denmark\newline}}
\email{joshuahunt@math.ku.dk}
\thanks{This work was supported by the Danish National Research Foundation through the Centre for Symmetry and Deformation (DNRF92) and the Copenhagen Centre for Geometry and Topology (DNRF151).}

\NewDocumentCommand{\SubgroupDecompositionWithTrivialSubgroup}{mm}{Let $\sC$ be a collection of subgroups of\, $#1$ that is closed under intersection and such that every elementary abelian $p$-subgroup of\, $#1$ is contained in a subgroup in $\sC$. There is an equivalence of symmetric monoidal ∞-categories
\[\StMod_{k#1} \rightshift{\simra}\lim_{#1/#2 \in \orbitC(#1)^\op} \StMod_{k#2}.\]}

\newcommand{\conditionForFactoringThroughPA}{Let $\cA$ and $\cB$ be small ∞-categories, $\cC$ be an ∞-category with all small colimits, and 
\[
\cA \xra{i} \cB \xra{F} \cC
\]
be functors with $i$ fully faithful. Let $\tilde F \colon \cP(\cB) \to \cC$ denote the left Kan extension of $\,F$ along the Yoneda embedding $\yoB$. There is a factorisation
\begin{equation*}
\begin{tikzcd}[ampersand replacement=\&]
\cP(\cB) \arrow[d, "i^*"'] \arrow[r, "\tilde F"] \& \cC \\
\cP(\cA) \arrow[ru, dashed, "\factorisation"']   \&    
\end{tikzcd}
\end{equation*}
if and only if\, $F$ is the left Kan extension of its restriction to $\cA$, \emph{i.e.} the natural map
\begin{equation*}
    \colim(i/b \to \cA \xra{Fi} \cC) \; \to \; F(b)
\end{equation*}
is an equivalence for every $b \in \cB$.}

\newcommand{\CentraliserAndNormaliserDecompositions}{Let $\sC$ be one of the collections $\SpG$, $\ApG$, $\BpG$, $\IpG$, or $\ZpG$. There is a subgroup decomposition
\[\StMod_{kG} \simra \lim_{G/P \in \orbitC(G)^\op} \StMod_{kP}\]
and a normaliser decomposition
\[\StMod_{kG} \simra \lim_{[\sigma] \in \osC(G)} \StMod_{kN_G(\sigma)}.\]
If $\,\sC\!$ is $\SpG$, $\ApG$, or $\ZpG$, then there is additionally a centraliser decomposition
\[\StMod_{kG} \simra \lim_{P \in \fusionC(G)} \StMod_{kC_G(P)}.\]}

\newcommand{\SInvarianceOfStMod}{The right Kan extension of\, $\StMod(-)\colon \orbit(G)^\op \to \lCati$ along the opposite of the Yoneda embedding $\orbit(G) \to \cS_G$ factors through the restriction map \[\cS_G^\op \simeq \cP(\orbit(G))^\op \to \cP(\porbit(G))^\op.\]
In particular, \,$\StMod(-)$ sends $S$-equivalences in $\cS_G$ to equivalences of ∞-categories.}

\begin{document}

\begin{abstract}
    We show that the stable module ∞-category of a finite group $G$ decomposes in three different ways as a limit of the stable module ∞-categories of certain subgroups of $G$. Analogously to Dwyer's terminology for homology decompositions, we call these the centraliser, normaliser, and subgroup decompositions. We construct centraliser and normaliser decompositions and extend the subgroup decomposition (constructed by Mathew) to more collections of subgroups. The key step in the proof is extending the stable module ∞-category to be defined for any $G$-space, then showing that this extension only depends on the $S$-equivariant homotopy type of a $G$-space. The methods used are not specific to the stable module ∞-category, so may also be applicable in other settings where an ∞-category depends functorially on $G$.
\end{abstract}

\pagestyle{normal}
\fancyhead[R]{\ifthenelse{\isodd{\value{page}}}{\thepage}{}}
\fancyhead[L]{\ifthenelse{\isodd{\value{page}}}{}{\thepage}}
\fancyhead[C]{\ifthenelse{\isodd{\value{page}}}{\textsc{Decompositions of the stable module ∞-category}}{\textsc{Joshua Hunt}}}

\maketitle

% \makeatletter
%     \providecommand\@dotsep{5}
% \makeatother
% \listoftodos\relax

\sectionspace
\section{Introduction}
\sectionspace

Let $G$ be a finite group and $p$ be a prime dividing the order of $G$. A homology decomposition of the classifying space $\B G$ is a diagram of spaces $F\colon D \to \cS$ such that, for every $d \in D$, the space $F(d)$ has the homotopy type of $\B H$ for some $H \leq G$, together with a map
\[
\hocolim F \to \B G
\]
that induces an isomorphism on mod $p$ homology. Homology decompositions have a long history in algebraic topology, with an early success being their use in \cite{JMO92} to classify self-maps of classifying spaces of compact, connected, simple Lie groups. They also played an important role in the classification of $p$-compact groups; see \cite{grodal10} for a survey. More recently, similar decomposition techniques have found applications in modular representation theory, for example in \cite{mathew16}, \cite{grodal18}, and \cite{BGH}.

In \cite{dwyer97}, Dwyer was able to give a unified treatment of three different types of homology decompositions for a fixed collection $\sC$ of subgroups of $G$ (where by the term \define{collection} we always mean a set of subgroups that is closed under conjugation by elements of $G$):
\begin{description}
    \item[Subgroup] let $D = \orbitC(G)$, the orbit category, and let $F\colon \orbitC(G) \to \cS$ take $G/H$ to $\B H$.
    \item[Centraliser] let $D = \fusionC(G)$, the fusion category, and let $F \colon \fusionC(G)^{\,\op} \to \cS$ take $H$ to $\B C_G(H)$.
    \item[Normaliser] let $D = \osC(G)$, the orbit simplex category, and let $F \colon \osC(G)^\op \to \cS$ take a simplex $\sigma = (H_0 < \ldots < H_n)$ to $\B N_G(\sigma)$, where $N_G(\sigma)$ denotes $\bigcap_{0 \leq i \leq n} N_G(H_i)$.
\end{description}
We recall the definition of the indexing categories below, in \cref{sec:notation-and-conventions}.
Dwyer showed that in all of these cases, $F$ provides a mod $p$ homology decomposition of $G$ if and only if the natural map $\sC_{\h G} \to (*)_{\h G} \simeq \B G$ induces an isomorphism on mod $p$ homology. Dwyer calls a collection $\sC$ satisfying this condition \define{ample}.

The stable module category $\StMod_{kG}$ of $G$ over a field $k$ of characteristic $p$ is obtained by \enquote{quotienting} the module category $\Mod_{kG}$ by the projective modules. In this paper, we show that analogues of the subgroup, centraliser and normaliser decompositions exist for $\StMod_{kG}$, viewed as an ∞-category, describing it in three different ways as a limit of ∞-categories $\StMod_{kH}$ for subgroups $H \leq G$. In this setting, Mathew \cite[Corollary~9.16]{mathew16} has already shown the existence of the subgroup decomposition for certain collections:
\begin{theorem*}[Mathew]
\SubgroupDecompositionWithTrivialSubgroup{G}{H}
\end{theorem*}
\noindent
Note that the change from a homotopy colimit (in Dwyer's result) to a homotopy limit (in Mathew's result) is due to the differing variances of the homotopy orbits functor $(-)_{\h G}$ and the stable module ∞-category functor $\StMod(-)$. Following ideas of Dwyer and others, we use $G$-spaces to encode decompositions of $\StMod(-)$: we formally Kan extend the functor 
\[\StMod(-)\colon \orbit(G)^\op \to \lCati\] 
to a functor defined on any $G$-space 
\begin{equation}
    \label{eq:extended-stmod}
    \StMod(-)\colon \cS_G^\op \to \lCati.
\end{equation}
This extended functor takes small homotopy colimits of $G$-spaces to homotopy limits of ∞-categories, so for any diagram of $G$-spaces $F\colon D \to \cS_G$, the canonical map $\hocolim(F) \to *$ induces a comparison map \begin{equation}
\label{eq:decomposition-for-general-F}
    \StMod_{kG} \simeq \StMod(*) \to \lim_{D^\op} \StMod(F(d)).
\end{equation} Dwyer constructed $G$-spaces (depending on the collection $\sC$) that encode the three homology decompositions listed above. For convenience, we will temporarily refer to these $G$-spaces as \define{encoding $G$-spaces}. Applying $\StMod(-)$ to the encoding $G$-spaces for $\sC$ gives functors as in \labelcref{eq:decomposition-for-general-F} that are candidates for the three decompositions of the stable module ∞-category.

To show that we do get a decomposition of the stable module ∞-category, we need to prove that the comparison functor \labelcref{eq:decomposition-for-general-F} is an equivalence. For this we use work of Grodal--Smith \cite{GS06}, which lists cases when certain canonical maps between the encoding $G$-spaces are $S$-equivalences, where $S$ is a Sylow $p$-subgroup. (Recall that that a map $f\colon X \to Y$ of $G$-spaces is an $S$-equivalence if it induces a homotopy equivalence $X^P \simra Y^P$ on $P$-fixed points for every $p$-subgroup $P \leq G$.) In \cref{sec:stmod-inverts-S-equivalences}, we prove that the extended functor \labelcref{eq:extended-stmod} inverts $S$-equivalences. In fact, we prove a more general result (\cref{prop:condition-for-factoring-through-P(A)}):
\begin{alphatheorem}
\label{prop:introduction's condition-for-factoring-through-P(A)}
\conditionForFactoringThroughPA
\end{alphatheorem}
Let $\porbit(G)$ denote the full subcategory of $\orbit(G)$ consisting of those $G$-sets whose isotropy groups are $p$-groups. From the above criterion, we deduce \cref{prop:S-invariance-of-StMod}:
\begin{alphatheorem}
\label{prop:introduction's S-invariance-of-StMod}
\SInvarianceOfStMod
\end{alphatheorem}
\noindent We note that the only property of $\StMod(-)$ used to deduce \cref{prop:introduction's S-invariance-of-StMod} from \cref{prop:introduction's condition-for-factoring-through-P(A)} is the existence of a subgroup decomposition, as in Mathew's result above. The approach therefore applies whenever an ∞-category depends functorially on $G$ and satisfies an analogous descent condition.

Given \cref{prop:introduction's S-invariance-of-StMod}, it is enough to find a zig-zag of $S$-equivalences from the encoding $G$-space for Mathew's subgroup decomposition to the encoding $G$-space for the decomposition that we are interested in. This problem was studied by Grodal--Smith in \cite{GS06} and we use their results in \cref{sec:zig-zags-of-S-equivalences} to obtain our main theorem (\cref{prop:centraliser-and-normaliser-decompositions}):
\begin{alphatheorem}
\CentraliserAndNormaliserDecompositions
\end{alphatheorem}

The collections mentioned in the theorem are defined as follows:
\begin{enumerate}
    \item $\SpG$ is the collection of non-trivial $p$-subgroups of $G$,
    \item $\ApG$ is the collection of non-trivial elementary abelian $p$-subgroups,
    \item $\BpG$ is the collection of non-trivial $p$-radical subgroups, \emph{i.e.} non-trivial $p$-subgroups $P \leq G$ such that $P$ is the maximal normal $p$-subgroup in $N_G(P)$,
    \item $\IpG$ is the collection of all non-trivial $p$-subgroups that are the intersection of a set of Sylow $p$-subgroups, and
    \item $\ZpG$ is the subcollection of $\ApG$ consisting of those $V$ such that $V$ is the set of elements in the centre of $C_G(V)$ whose order divides $p$, \emph{i.e.} such that $V = \Omega_1O_pZ(C_G(V))$, using standard group-theoretic notation such as found in \cite{aschbacher00}.
\end{enumerate}

\begin{remark}
Neither $\BpG$ nor $\ZpG$ are closed under intersections, so the subgroup decompositions for $\BpG$ and $\ZpG$ in \cref{prop:centraliser-and-normaliser-decompositions} are new and do not follow immediately from Mathew's subgroup decomposition. For example, if $G = \operatorname{PSL}_3(7)$, then there are rank two elementary abelian subgroups \[\begin{pmatrix} 1 & * & * \\
                         & 1 & 0 \\
                         &  & 1 
        \end{pmatrix} \quad \text{and} \quad
        \begin{pmatrix} 1 & 0 & * \\
                         & 1 & * \\
                         &  & 1
        \end{pmatrix}\]
that are contained in both $\BpG$ and $\ZpG$, but whose intersection
\[\begin{pmatrix} 1 & 0 & * \\
                         & 1 & 0 \\
                         &  & 1
        \end{pmatrix}\]
is contained in neither collection.
\end{remark}

\begin{acknowledgements}
This work was supported by the Danish National Research Foundation through the Centre for Symmetry and Deformation (DNRF92) and the Copenhagen Centre for Geometry and Topology (DNRF151). I would like to thank Markus Land and Piotr Pstrągowski for useful conversations about the contents of this paper (in particular, Markus suggested a way to prove \cref{prop:condition-for-factoring-through-P(A)}); my supervisor, Jesper Grodal, for suggesting the problem and his advice and encouragement; and Kaif Hilman, for catching a mistake in a previous draft as well as many other useful comments.
\end{acknowledgements}

\sectionspace
\section{Notation, conventions, and background}
\sectionspace
\label{sec:notation-and-conventions}
Throughout the paper, $G$ will refer to a finite group and $p$ will be a fixed prime dividing the order of $G$. We let $k$ be a field of characteristic $p$. A \define{collection} of subgroups of $G$ is a set of subgroups that is closed under conjugation by elements of $G$.

We let $\orbit(G)$ denote the \define{orbit category} of $G$, whose objects are transitive left $G$-sets and whose morphisms are $G$-equivariant maps between them. For any collection $\sC$ of subgroups of $G$, we let $\orbitC(G)$ denote the full subcategory of $\orbit(G)$ on those $G$-sets whose isotropy subgroups are contained in $\sC$. Note that every object of $\orbit(G)$ is isomorphic to a $G$-set of the form $G/H$, with $H$ a subgroup of $G$, and that $G/H$ lies in $\orbitC(G)$ if and only if $H \in \sC$. With this identification, the morphisms in the orbit category are related to subconjugation relations between subgroups:
\[
\Hom_{\orbit(G)}(G/H, G/K) \: \cong \: \{\,g \in G : H^g \leq K\,\}/K.
\]
When $\sC$ is the collection of all non-trivial $p$-subgroups of $G$, we will use the notation $\orbitS(G)$ instead of $\orbitC(G)$. When we wish to additionally include the trivial subgroup in a collection, we will add the superscript \enquote{$e$}, for example writing $\porbit(G)$. (In other parts of the literature, a superscript \enquote{$*$} was used to indicate the \emph{removal} of the trivial subgroup, a convention we do not use here.)

We let $\fusionC(G)$ denote the \define{fusion category} of $G$, whose objects are the subgroups in $\sC$ and whose morphisms are homomorphisms that are induced by conjugation by an element of $G$.

We let $\osC(G)$ denote the \define{orbit simplex category} of $G$, which is the poset of $G$-conjugacy classes of non-empty chains $\sigma = (H_0 < \ldots < H_n)$ of subgroups in $\sC$, ordered by refinement: that is, we have $[\sigma] \leq [\tau]$ if we can find representatives $\sigma$ and $\tau$ for the conjugacy classes such that $\sigma \subseteq \tau$. The objects of $\osC(G)$ identify with the $G$-conjugacy classes of non-degenerate simplices of the nerve of $\sC$.

Since we deal exclusively with homotopy (co)limits, we will drop the adjective \enquote{homotopy} here: when we refer to a \enquote{colimit of $G$-spaces} we will implicitly mean a homotopy colimit. We follow Lurie's convention of using the prefix \enquote{∞-} instead of \enquote{$(\text{∞},1)$-}. We use $\lCati$ to denote the ∞-category of large ∞-categories, which has all large limits (though we will only need the existence of small limits). Let $\FunColim(\cC, \cD)$ denote the full subcategory of $\Fun(\cC, \cD)$ spanned by the functors that preserve small colimits.

The \define{stable module ∞-category} $\StMod_{kG}$ is defined as the localisation of the module category $\Mod_{kG}$ at the \define{stable equivalences}, \emph{i.e.} at those maps $f\colon M \to N$ and $g\colon N \to M$ such that $fg - \id_N$ and $gf - \id_M$ both factor through a projective module. The homotopy category of the stable module ∞-category has been studied by representation theorists: for example, Benson--Iyengar--Krause \cite{BIK11} use it as a way of classifying $kG$-modules when $\Mod_{kG}$ has wild representation type. The objects of the homotopy category are $kG$-modules and the hom sets are given by 
\[
\pi_0 \Map_G(M, N) \cong \Hom_G(M, N)/(f \sim 0\text{ if $f$ factors through a projective}).
\] 
The stable module ∞-category is a presentable, stable, symmetric monoidal ∞-category. See \cite[Section~5]{carlson96} for a discussion of the homotopy category and \cite[Definition~2.2]{mathew15} for a construction of $\StMod_{kG}$ as an ∞-category.

In \cite[Section~9.5]{mathew16}, Mathew constructs a functor \[\StMod(-)\colon \orbit(G)^\op \to \CAlg(\PrLst)\] whose value on $G/H$ is equivalent to $\StMod_{kH}$. Here $\PrLst$ denotes the ∞-category of presentable, stable ∞-categories and left adjoint functors between them. This functor will play a crucial role for us, so we spend the rest of this section describing its construction. For any $H \leq G$, we have a symmetric monoidal restriction functor $\Res_H^G\colon \StMod_{kG} \to \StMod_{kH},$ whose right adjoint $\CoInd_H^G\colon \StMod_{kH} \to \StMod_{kG}$ is consequently lax symmetric monoidal. This implies that $\CoInd_H^G(k)$ is a commutative algebra object of $\StMod_{kG}$, which we will denote $A^G_H$. The underlying module of $A^G_H$ is $\prod_{G/H} k$ with its permutation action. 

In \cite[Theorem~1.2]{balmer15}, Balmer proves that the homotopy category of $\StMod_{kH}$ is equivalent to the category of modules over $A^G_H$ internal to the homotopy category of $\StMod_{kG}$. Proposition~9.12 of \cite{mathew16} generalises this result to the ∞-category $\StMod_{kH}$:
\begin{theorem}[Balmer, Mathew]\label{prop:restriction-as-base-change}
There is a natural symmetric monoidal equivalence 
\[
\StMod_{kH} \simra \Mod_{\StMod_{kG}}(A^G_H)
\]
induced by coinduction, under which the free/forget adjunction $\StMod_{kG} \fallbackleftrightarrows \Mod_{\StMod_{kG}}(A^G_H)$ corresponds to the restriction/coinduction adjunction $\StMod_{kG} \fallbackleftrightarrows \StMod_{kH}$.
\end{theorem}

We have a functor
\begin{IEEEeqnarray*}{CCC}
\orbit(G)^\op & \to & \CAlg(\StMod_{kG}) \\
G/H & \mapsto & A^G_H
\end{IEEEeqnarray*}
that on underlying modules sends a morphism $G/H \to G/K$ to the pullback map $\prod_{G/K} k \to \prod_{G/H} k$.
This functor can be constructed by composing the analogous functor $\orbit(G)^\op \to \CAlg(\Mod_{kG})$ with the localisation functor $\Mod_{kG} \to \StMod_{kG}$.
% Mod_{kG} is a symmetric monoidal model category, so this preserves algebra objects
\cref{prop:restriction-as-base-change} implies that we obtain a functor 
\label{sec:stmod-is-functor-from-orbit-category}
\[
\StMod(-)\colon \orbit(G)^\op \to \lCati
\]
that takes $G/H$ to an ∞-category equivalent to $\StMod_{kH}$. By using the description given in \cite[Construction~5.23]{MNN17} of the inverse to the equivalence in \cref{prop:restriction-as-base-change}, one can check that a morphism $G/H \to G/K$ in $\orbit(G)$ is sent to the restriction functor $\StMod_{kK} \to \StMod_{kH}$. Since both $\StMod_{kH}$ and $\Res^K_H\colon \StMod_{kK} \to \StMod_{kH}$ lie in $\CAlg(\PrLst)$, we have constructed a functor \[\orbit(G)^\op \to \CAlg(\PrLst)\] as desired.

\sectionspace
\section{The stable module \texorpdfstring{∞}{infinity}-category of a \texorpdfstring{$G$}{G}-space}
\sectionspace

Let $\cS$ denote the ∞-category of small spaces. Recall that the ∞-category $\cS_G$ of small $G$-spaces is equivalent to $\Fun(\orbit(G)^\op, \cS)$ and the Yoneda embedding $\yo\colon \orbit(G) \to \cS_G$ identifies with the inclusion of $\orbit(G)$ as the transitive, discrete $G$-spaces.

The stable module ∞-category determines a functor $\StMod(-) \colon \orbit(G)^\op \to \lCati$ that takes $G/H$ to $\StMod_{kH}$. We can formally extend this functor to $G$-spaces by right Kan extension along the opposite of the Yoneda embedding:
\begin{equation*}
\begin{tikzcd}
\orbit(G)^\op \arrow[d, "\yo"'] \arrow[r] & \lCati \\
\cS_G^\op \arrow[ru, dashed, "\StMod(-)"']    &      
\end{tikzcd}
\end{equation*}
Since $\orbit(G)$ is small, the existence of small limits in $\lCati$ guarantees the existence of the Kan extension, which we still denote by $\StMod(-)$. It sends small (homotopy) colimits of $G$-spaces to limits in $\lCati$.

\begin{remark}
We could equally well have taken $\StMod(-)$ to be a functor to $\CAlg(\PrLst)$ while carrying out the above construction, since $\CAlg(\PrLst)$ also has small limits and the composition $\CAlg(\PrLst) \to \PrLst \to \lCati$ preserves limits \cite[Corollary~3.2.2.4]{HA}. In other words, the \enquote{stable module ∞-category of a $G$-space} is presentable, stable, and has a symmetric monoidal structure that preserves finite colimits in each variable. We will not need these facts, but mention them to point out that the extended definition of the stable module ∞-category shares many features with the standard definition.

%\cite[Remark 4.8.1.8]{HA}.
% \cite[Corollary 3.2.2.4, Proposition 4.8.2.18]{HA} and \cite[Proposition 5.5.3.13]{HTT}
\end{remark}

\sectionspace
\section{Factorising Kan extensions}
\sectionspace
\label{sec:the-general-case}

Our first goal is to show that $\StMod(-)$ only sees the $S$-equivariant homotopy type of a $G$-space: that is, we have a factorisation
\begin{equation*}
\begin{tikzcd}
\cS_G^\op \simeq \cP(\orbit(G))^\op \arrow[d, "i^*"] \arrow[r]    &  \lCati\\
\phantom{\cS_G^\op = }\cP(\porbit(G))^\op \arrow[ru, dashed] &
\end{tikzcd}
\end{equation*}
where $i\colon \porbit(G) \hookrightarrow \orbit(G)$ is the inclusion of the full subcategory of transitive $G$-sets with $p$-group isotropy and $\cP(\cA)$ is the ∞-category $\Fun(\cA^\op, \cS)$ of presheaves on $\cA$.

For simplicity of notation, we dualise and consider the following question:
\begin{question}
\label{question:factorisation-through-P(A)} Let $\cA$ and $\cB$ be small ∞-categories, $\cC$ be an ∞-category with all small colimits, and 
\[\cA \xra{i} \cB \xra{F} \cC\]
be functors with $i$ fully faithful. 
%\tilde F exists by HTT 4.2.3.13
When does the left Kan extension $\tilde F$ of $F$ along the Yoneda embedding $\yoB$ admit a factorisation through $\cP(\cA)$ as indicated in the diagram below?
\begin{equation*}
%\label[diagram]{diag:factoring-tilde-F}
\begin{tikzcd}
\cB \arrow[d, "\yoB"'] \arrow[dr, "F"] & \\
\cP(\cB) \arrow[d, "i^*"'] \arrow[r, "\tilde F", near start] & \cC \\
\cP(\cA) \arrow[ru, dashed, "\factorisation"'] &
\end{tikzcd}
\end{equation*}
\end{question}

We answer this question in \cref{prop:condition-for-factoring-through-P(A)}, providing a necessary and sufficient condition on $F$ for such a factorisation $\factorisation$ to exist: namely, $F$ must be equivalent to the left Kan extension of $Fi$ along $i$.

\begin{notation}
\label{notation:i-and-j}
In this situation, there are two functors that are induced by restriction along $i$ (or its opposite) and hence could reasonably be denoted by $i^*$, namely
\[
\cP(\cB) \to \cP(\cA) \text{\quad and \quad} \Fun(\cB, \cC) \to \Fun(\cA, \cC).
\]
We will denote the first functor by $i^*$ and the second functor instead by $\jpb$. We hope that this prevents more confusion than it causes. Both of these functors have left adjoints given by left Kan extension, %these exist by HTT 4.3.2.14, 4.3.2.16 and 4.3.2.17
which we will write as
\[
i_!\colon \cP(\cA) \to \cP(\cB) \text{\quad and \quad} \jshr\colon \Fun(\cA, \cC) \to \Fun(\cB, \cC).
\]
We will similarly use $(\yoB)_!$ to denote left Kan extension along the Yoneda embedding. % This exists by HTT 5.1.5.5 and 5.1.5.6
The adjunction $i_! \dashv i^*$ induces an adjunction 
\begin{equation*}
\begin{tikzcd}
\Fun(\cP(\cA),\cC) \arrow[d, shift right=0.25cm, "(i^*)^*"'] \arrow[d, phantom, "\dashv"]\\
\Fun(\cP(\cB), \cC). \arrow[u, shift right=0.25cm, "(i_!)^*"']
\end{tikzcd}
\end{equation*}
\end{notation}

\begin{remark}
\label{rem:outline-of-factorisation-proof}
We briefly summarise the proof of \cref{prop:condition-for-factoring-through-P(A)}, making forward reference to lemmas that we will prove later in the \lcnamecref{sec:the-general-case}. We will show that the following statements are all equivalent:
\begin{enumerate}
    \item $\tilde F$ factors through $i^*$.
    \item The natural transformation $\tilde F i_! i^* \to \tilde F$ induced by the counit of the $i_! \dashv i^*$ adjunction is an equivalence.
    \item $\tilde F$ is equivalent, via the counit of the $(i^*)^* \dashv (i_!)^*$ adjunction, to the composition $(i^*)^* \circ (i_!)^* \circ (\yoB)_!$ applied to $F$.
    \item $\tilde F$ is equivalent, via the counit of the $\jshr \dashv \jpb$ adjunction, to the composition $(\yoB)_! \circ \jshr \circ \jpb$ applied to $F$.
    \item The natural transformation $\jshr \jpb F \to F$ induced by the counit of the $\jshr \dashv \jpb$ adjunction is an equivalence.
\end{enumerate}

The equivalence of (i) and (ii) is \cref{prop:first-factorisation-condition}. Since $\tilde F$ is the left Kan extension of $F$ along $\yoB$, we see that (iii) is just a rewriting of (ii) with different notation. \cref{prop:composing-lans,prop:second-Kan-lemma} show that (iii) is equivalent to (iv). Checking that the natural transformation in (iv) is still induced by the counit is a straightforward but tiring diagram chase that we include in \cref{sec:its-a-counit}. 
Finally, \cite[Theorem~5.1.5.6]{HTT} shows that restriction along the Yoneda embedding induces an equivalence 
\[
\FunColim(\cP(\cB), \cC) \simra \Fun(\cB, \cC),
\]
so (iv) is equivalent to (v).
\end{remark}

The rest of the \lcnamecref{sec:the-general-case} fills in the details omitted in \cref{rem:outline-of-factorisation-proof}. We begin by showing that if $\tilde F$ does factor through $i^*$, then such a factorisation is unique.

\begin{lemma} 
\label{prop:first-factorisation-condition}
Let $H\colon \cP(\cB) \to \cC$ be a functor. Any functor $\overline{H}\colon \cP(\cA) \to \cC$ that satisfies $H \simeq \overline{H} \circ i^*$ must be given by $\overline{H} \simeq H \circ i_!$. Furthermore, $H$ factors through $i^*$ if and only if the natural transformation 
\begin{equation*}
\label{eq:tilde-F-epsilon}
H \varepsilon \colon H i_! i^* \to H
\end{equation*}
induced by the counit of the $i_! \dashv i^*$ adjunction is an equivalence in $\Fun(\cP(\cB), \cC)$.
\end{lemma}
\begin{proof}
Since $i$ is fully faithful, the unit of the adjunction induces an equivalence $\id \simra i^*i_!$. Therefore, $H i_! \simeq \overline{H} i^* i_! \simeq \overline{H}$. The second claim is a straightforward check using naturality of the equivalence $H \simeq \overline{H} \circ i^*$ and the triangle identities for $i_! \dashv i^*$.
\end{proof}

The condition in \cref{prop:first-factorisation-condition} for a factorisation to exist is in terms of $\tilde F$, so our next goal is to rewrite this as a condition on $F$. We will need two lemmas regarding properties of Kan extensions.

\begin{lemma}
\label{prop:composing-lans}
There is a canonical equivalence
\[
Fi \simra \tilde F i_! \yoA,
\]
which by the universal property of left Kan extension induces a natural transformation
\[
\Lan_{\yoA}(F i) \to \tilde F i_!
\]
as functors $\cP(\cA) \to \cC$. This natural transformation is an equivalence; that is, 
\[
(\yoA)_! \circ \jpb \simra (i_!)^* \circ (\yoB)_!
\]
as functors $\Fun(\cB, \cC) \to \FunColim(\cP(\cA), \cC)$.
\end{lemma}
\begin{proof}
Since $\yoB$ is fully faithful, we have a commutative diagram
\begin{equation*}
\begin{tikzcd}
\cA \arrow[r, "i"] \arrow[d, "\yoA"'] & \cB \arrow[r, "F"] \arrow[d, "\yoB"'] & \cC \\
\cP(\cA) \arrow[r, "i_!"] & \cP(\cB) \arrow[ur, "\tilde F \coloneqq \Lan_{\yoB}(F)"']
\end{tikzcd}
\end{equation*}
that gives rise to the first equivalence.

%The square commutes by HTT 5.2.6.3, which identifies $i_!$ with $\cP(i)$.
The three functors $\tilde F$, $i_!$, and $\Lan_{\yoA}(F i)$ all preserve small colimits, so it is enough to check that they restrict along the Yoneda embedding $\yoA$ to equivalent functors in $\Fun(\cA, \cC)$, by \cite[Theorem~5.1.5.6]{HTT}. We then observe that $\Lan_{\yoA}(Fi)$ restricts to $Fi$, because $\yoA$ is also fully faithful.
\end{proof}

\begin{lemma}
\label{prop:second-Kan-lemma}
Let $H\colon \cA \to \cC$ be a functor. There is a natural equivalence
\[
\Lan_{\yoA}(H) \circ i^* \simeq \Lan_{\yoB i}(H)
\]
as functors $\cP(\cB) \to \cC$. That is,
\[
(i^*)^* \circ (\yoA)_! \simeq (\yoB)_! \circ \jshr
\]
as functors $\Fun(\cA, \cC) \to \FunColim(\cP(\cB), \cC)$.
\end{lemma}
\begin{proof}
We have a commutative diagram
\begin{equation*}
\begin{tikzcd}
\cA \arrow[r, "i"] \arrow[d, "\yoA"'] & \cB \arrow[d, "\yoB"']\\
\cP(\cA) \arrow[r, "i_!"] & \cP(\cB)
\end{tikzcd}
\end{equation*}
that induces a commutative diagram of pullback functors:
\begin{equation}
\label[diagram]{diag:pullback-functors}
    \begin{tikzcd}
        \Fun(\cA, \cC)  & \Fun(\cB, \cC)  \arrow[l, "\jpb"'] \\
 \FunColim(\cP(\cA), \cC) \arrow[u, "(\yoA)^*"'] & \arrow[l, "(i_!)^*"'] \FunColim(\cP(\cB), \cC)  \arrow[u, "(\yoB)^*"']\
    \end{tikzcd}
\end{equation}
The adjunction $i_! \dashv i^*$ induces an adjunction $(i^*)^* \dashv (i_!)^*$, so
%\todo{Reference for this? It's Proposition 23.11 in Markus' notes.}
all of the functors in the above diagram have left adjoints. Thus, we obtain another commutative diagram:
\begin{equation*}
%\label[diagram]{diag:shriek-functors}
    \begin{tikzcd}
        \Fun(\cA, \cC)  \arrow[r, "\jshr"]  \arrow[d, "(\yoA)_!"] & \Fun(\cB, \cC) \arrow[d, "(\yoB)_!"]\\
  \FunColim(\cP(\cA), \cC) \arrow[r, "(i^*)^*"] & \FunColim(\cP(\cB), \cC)
    \end{tikzcd}
\end{equation*}
This is what we aimed to prove.
\end{proof}

Combining the above lemmas, we deduce:
\begin{theorem}
\label{prop:condition-for-factoring-through-P(A)}
\conditionForFactoringThroughPA
\end{theorem}
\begin{proof}
In \cref{prop:first-factorisation-condition} we showed that $\tilde F$ factors through $i^*$ if and only if the natural transformation
\begin{equation*}
\tilde F i_! i^* \to \tilde F
\end{equation*}
induced by the counit of the $i_! \dashv i^*$ adjunction is an equivalence. By \cref{prop:composing-lans}, the left-hand side of this is equivalent to $\Lan_{\yoA}(Fi) \circ i^*$, while by \cref{prop:second-Kan-lemma}, this in turn is equivalent to $\Lan_{\yoB i}(Fi)$. Since both sides of 
\begin{equation}
\label{eq:factorisation-condition}
    \Lan_{\yoB i}(Fi) \to \tilde F
\end{equation}
preserve colimits in $\cP(\cB)$, we can check whether \labelcref{eq:factorisation-condition} is an equivalence after restricting along $\yoB$. Therefore, $\tilde F$ factors through $i^*$ if and only if the natural transformation
\[
\Lan_i(Fi) \to F
\]
is an equivalence.
\end{proof}

\sectionspace
\section{The stable module \texorpdfstring{∞-}{infinity-}category is \texorpdfstring{$S$}{S}-homotopy invariant}
\sectionspace
\label{sec:stmod-inverts-S-equivalences}

We can now return to the specific case that interests us, namely
\begin{equation*}
\porbit(G)^\op \xra{i} \orbit(G)^\op \xra{\StMod(-)} \lCati.
\end{equation*}
We wish to show that we have an induced functor $\factorisation\colon \cP(\porbit(G))^\op \to \lCati$. By the dual of \cref{prop:condition-for-factoring-through-P(A)}, this happens if and only if the natural map
\begin{equation*}
    \StMod(G/H) \; \to \; \lim\big((i/(G/H))^\op \to \porbit(G)^\op \xra{\StMod} \lCati\big)
\end{equation*}
is an equivalence for every $H \leq G$. The slice category $i/(G/H)$ is naturally equivalent to the $p$-orbit category $\porbit(H)$, with the functor $\porbit(H) \to \porbit(G)$ being given by $H/P \mapsto G/P$. Mathew \cite[Corollary~9.16]{mathew16} proves a subgroup decomposition of $\StMod_{kH}$ over the orbit category:

\begin{theorem}[Mathew]
\label{prop:subgroup-decomposition-with-trivial-subgroup}
\SubgroupDecompositionWithTrivialSubgroup{H}{K}
\end{theorem}

\noindent In light of \cref{prop:restriction-as-base-change}, it is therefore enough to transport the above decomposition, applied to $\sC = \cS_p(H) \cup \{1\}$, up to $\StMod_{kG}$:

\begin{lemma}
\label{prop:module-subgroup-decomposition}
The natural map
\begin{equation*}
    \StMod(G/H) \rightshift{\simra} \lim_{H/P \in \porbit(H)^\op} \StMod(G/P)
\end{equation*}
is an equivalence.
\end{lemma}
\begin{proof}
Let $\cM$ denote $\Mod_{\StMod_{kG}}(A^G_H)$; recall from \cref{sec:stmod-is-functor-from-orbit-category} that $A^G_H$ is $\CoInd^G_H(k)$ and that $\StMod(G/H)$ is equal to $\cM$ by definition. By \cite[Corollary~3.4.1.9]{HA}, for any $A \in \CAlg(\cM)$ we have a natural equivalence
\begin{equation*}
    \Mod_{\cM}(A) \simra \Mod_{\StMod_{kG}}(A)
\end{equation*}
given by forgetting the $A^G_H$-module structure. Therefore, we wish to prove that
\begin{equation*}
    \cM \rightshift{\to} \lim_{H/P \in \porbit(H)^\op} \Mod_{\cM}(A^G_P)
\end{equation*}
is an equivalence.

Recall from \cref{prop:restriction-as-base-change} that coinduction induces a functor 
\[
\barcoind \colon \StMod_{kH} \simra \cM
\]
that is an equivalence of symmetric monoidal ∞-categories. We obtain a commutative diagram
\begin{equation*}
\begin{tikzcd}[row sep=0.7cm]
\cM \arrow[r] & \displaystyle\lim_{H/P \in \porbit(H)^\op} \Mod_{\cM}(A^G_P) \\
              & \mathrlap{\displaystyle\lim_{H/P \in \porbit(H)^\op} \Mod_{\cM}(\CoInd_H^G(A^H_P))}\phantom{\displaystyle\lim_{H/P \in \porbit(H)^\op} \Mod_{\cM}(A^G_P)} \arrow{u}[anchor=center,rotate=-90,yshift=-1.0ex]{\sim} \\
\StMod_{kH} \arrow[uu, "\barcoind"] \arrow[r, "\sim"] & \mathrlap{\displaystyle\lim_{H/P \in \porbit(H)^\op} \Mod_{\StMod_{kH}}(A^H_P)}\phantom{\displaystyle\lim_{H/P \in \porbit(H)^\op} \Mod_{\cM}(A^G_P)} \arrow[u, "\barcoind"]
\end{tikzcd}
\end{equation*}
whose bottom arrow is an equivalence by Mathew's \cref{prop:subgroup-decomposition-with-trivial-subgroup}.
\end{proof}

We have therefore established:
\begin{theorem}
\label{prop:S-invariance-of-StMod}
\SInvarianceOfStMod
\end{theorem}

\begin{remark}
As noted in the introduction, the only part of this argument that was non-formal was checking the descent statement in \cref{prop:module-subgroup-decomposition}. Therefore, given a collection $\sF$ of subgroups of $G$ that is closed under intersections and a functor $F\colon \orbit(G)^\op \to \lCati$ that has a subgroup decomposition associated with $\sF$, the extended functor $\tilde F\colon \cS_G^\op \to \lCati$ inverts the weak equivalences associated with $\sF$.
\end{remark}

\sectionspace
\section{Decompositions of the stable module \texorpdfstring{∞-}{infinity-}category}
\sectionspace
\label{sec:zig-zags-of-S-equivalences}

In this section, we recall Dwyer's construction \cite[Sections~3.4--3.7]{dwyer98} of $G$-spaces that encode candidates for the subgroup, centraliser, and normaliser decompositions. These $G$-spaces are (homotopy) colimits, so applying $\StMod(-)$ gives a limit of ∞-categories that receives a comparison map from $\StMod_{kG}$; 
the $G$-space encodes a decomposition precisely when this map is an equivalence. We also have $S$-homotopy equivalences between these $G$-spaces, so can use the fact that $\StMod(-)$ inverts $S$-homotopy equivalences of $G$-spaces to \enquote{propagate} Mathew's subgroup decomposition to centraliser and normaliser decompositions.

Let $\sC$ be a collection of subgroups of $G$. Note that we can consider a $G$-set as a discrete $G$-space.
\begin{description}
    \item[Centraliser] Recall from \cref{sec:notation-and-conventions} that the fusion category $\fusionC(G)$ has objects given by the elements of $\sC$ and morphisms given by group homomorphisms that are induced by conjugation in $G$. We have a functor $\alpha \colon \fusionC(G)^{\,\op} \to \cS_G$ that takes $H \in \sC$ to the conjugacy class of the inclusion $H \hookrightarrow G$, which is isomorphic to $G/C_G(H)$ as a $G$-set. The colimit of $\alpha$ is a $G$-space that we will denote $\E\fusionC(G)$.
    \item[Subgroup] We have an inclusion functor $\beta \colon \orbitC(G) \to \cS_G$. We let $\E\orbitC(G)$ denote the colimit of $\beta$.
    \item[Normaliser] Recall from \cref{sec:notation-and-conventions} that the orbit simplex category $\osC(G)$ is the poset of $G$-conjugacy classes of non-degenerate simplices in the nerve of $\sC$, ordered by refinement. We have a functor $\delta\colon \osC(G)^{\op} \to \cS_G$ that takes a $G$-orbit of simplices $[\sigma]$ to itself, considered as a discrete $G$-space. This $G$-set is isomorphic to $G/N_G(\sigma)$, where for $\sigma = (P_0 < \ldots < P_n)$ we define $N_G(\sigma) \coloneqq \bigcap_{0 \leq i \leq n} N_G(P_i)$. We let $\sd(\sC)$ denote the colimit of $\delta$.
\end{description}

\begin{remark} Note that $\colim(\delta)$ really is $G$-equivalent to the subdivision of the nerve of the poset $\sC$, justifying our notation. This can be checked directly using the model for $\colim \delta$ given by the diagonal of a bisimplicial set that is explained in \cite[\nopp XII~5.2]{BK72}.
\end{remark}

\begin{remark}
Dwyer calls these three spaces $X^\alpha_\sC$, $X_\sC^\beta$, and $\sd X^\delta_\sC$, respectively. For historical reasons, $\E\fusionC(G)$ is sometimes called $\E\mathbf{A}_\sC$ elsewhere in the literature.
\end{remark}

Since $\StMod(-)$ sends small colimits of $G$-spaces to limits of ∞-categories, we get
\begin{IEEEeqnarray*}{RCL}
\StMod(\E\fusionC(G)) & \rightshift{\simeq} & \lim_{H \in \fusionC(G)} \StMod(\alpha(H)) \\
                     & \rightshift{\simeq} & \lim_{H \in \fusionC(G)} \StMod(G/C_G(H)) \\
                     & \rightshift{\simeq} & \lim_{H \in \fusionC(G)} \StMod_{kC_G(H)}.
\end{IEEEeqnarray*}
The natural map $\E\fusionC(G) \to *$ induces a restriction functor $\StMod_{kG} \to \lim_{\fusionC(G)} \StMod_{kC_G(H)}$, so in this way $\E\fusionC(G)$ encodes a candidate for a centraliser decomposition for $\StMod_{kG}$. Similarly, $\E\orbitC(G) \to *$ induces a functor
\begin{equation*}
\StMod_{kG} \rightshift{\to} \lim_{H \in \orbitC(G)^{\op}} \StMod_{kH}    
\end{equation*}
corresponding to a subgroup decomposition and $\sd(\sC) \to *$ induces a functor
\begin{equation*}
\StMod_{kG} \rightshift{\to} \lim_{\sigma \in \osC(G)} \StMod_{kN_G(\sigma)}
\end{equation*}
corresponding to a normaliser decomposition. However, for a general collection $\sC$ there is no reason for any of these functors from $\StMod_{kG}$ to be an equivalence. We have comparison maps between the $G$-spaces associated with a collection:
\begin{equation*}
\begin{tikzcd}
    & \sd(\sC) \arrow[sloped, d, "\sim"] & \\
\E\orbitC(G) \arrow[r] & \sC & \arrow[l] \E\fusionC(G) 
\end{tikzcd}
\end{equation*}
Here the map from $\sd(\sC)$ is the $G$-equivalence sending a simplex $P_0 < \ldots < P_n$ to $P_0$; the map from $\E\orbitC(G)$ sends a point $x \in G/H$ to its stabiliser $G_x$; and the map from $\E \fusionC(G)$ sends an inclusion $H \hookrightarrow G$ to $H$.

Let $\SpG$ denote the collection of all non-trivial $p$-subgroups and $\ApG$ denote the subcollection of non-trivial elementary abelian $p$-subgroups. Recall that Mathew's \cref{prop:subgroup-decomposition-with-trivial-subgroup} shows that for certain collections, including $\sC = \SpG \cup \{1\}$ and $\sC = \ApG \cup \{1\}$, we obtain a subgroup decomposition. However, to obtain a useful centraliser or normaliser decomposition, we need to remove the trivial subgroup from the collection, otherwise $\StMod_{kG}$ itself will appear in the decomposition on the right hand side. We therefore need a minor variation of \cref{prop:subgroup-decomposition-with-trivial-subgroup}:

\begin{lemma}
\label{prop:subgroup-decomposition-without-trivial-subgroup}
Let $\sC$ be a collection of subgroups of\, $G$ such that $\sC \cup \{1\}$ is closed under intersection and such that every elementary abelian $p$-subgroup of\, $G$ is contained in a subgroup in $\sC$. There is an equivalence of symmetric monoidal ∞-categories
\[\StMod_{kG} \rightshift{\simra} \lim_{G/H \in \orbitC(G)^\op} \StMod_{kH}.\]
\end{lemma}
\begin{proof}
Since $\StMod(G/\{1\}) \simeq *$, the diagram of ∞-categories that includes the trivial subgroup is a right Kan extension of the diagram that omits it, so the limits of the two diagrams agree.
\end{proof}

\Cref{prop:subgroup-decomposition-without-trivial-subgroup} shows we have a subgroup decomposition for $\sC = \SpG$ and $\sC = \ApG$. We now transfer this result to other collections and decompositions using the fact that $\StMod(-)$ inverts $S$-equivalences. We consider the following collections of $p$-subgroups: let $\BpG$ be the collection of non-trivial $p$-radical subgroups, \emph{i.e.} non-trivial $p$-subgroups $P \leq G$ such that $P$ is the maximal normal $p$-subgroup in $N_G(P)$. Let the collection $\IpG$ consist of all non-trivial subgroups that are intersections of a set of Sylow $p$-subgroups in $G$. Finally, let $\ZpG$ be the subcollection of $\ApG$ consisting of those subgroups $V$ such that $V$ is the set of elements in the centre of $C_G(V)$ whose order divides $p$, \emph{i.e.} such that $V = \Omega_1O_pZ(C_G(V))$.

\begin{theorem}
\label{prop:centraliser-and-normaliser-decompositions}
\CentraliserAndNormaliserDecompositions
\end{theorem}
\begin{proof}
We can restate the conclusion of \Cref{prop:subgroup-decomposition-without-trivial-subgroup} for $\SpG$ and $\ApG$ as saying that
\begin{equation}
\label{eq:subgroup-decomposition-spaces}
\E\orbitS(G) \to * \quad \text{and} \quad \E\orbitA(G) \to *
\end{equation}
are both sent to equivalences by $\StMod(-)$.

We now transport this information along $S$-homotopy equivalences. The following table is taken from \cite[Theorem~1.1]{GS06}; a solid line denotes a $G$-homotopy equivalence, while a dashed line denotes an $S$-homotopy equivalence. The column labels represent the different collections of subgroups, where for conciseness we omit $p$ and $G$ from the notation.

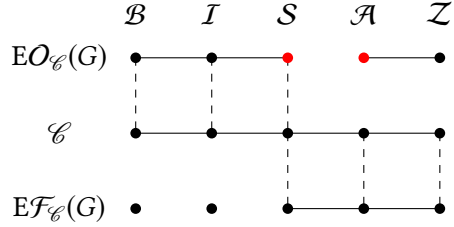
\begin{equation*}
\vspace{0.2cm}
\begin{tikzpicture}[Sequiv/.style={dashed}]
    \draw[anchor=south] (0, 0.3) node {$\cB$};
    \draw[anchor=south] (1, 0.3) node {$\cI$};
    \draw[anchor=south] (2, 0.3) node {$\cS$};
    \draw[anchor=south] (3, 0.3) node {$\cA$};
    \draw[anchor=south] (4, 0.3) node {$\cZ$};
    
    \draw (-1, 0) node {$\E\orbitC(G)$};
    \draw (-1, -1) node {$\sC$};
    \draw (-1, -2) node {$\E\fusionC(G)$};
    
    \draw (0,-1) -- (4,-1);
    \draw (0,0) -- (2,0);
    \draw (3,0) -- (4,0);
    \draw (2,-2) -- (4,-2);
    
    \foreach \x in {0,1,2}{
        \draw[Sequiv] (\x,0) -- (\x,-1);
    }
    \foreach \x in {2,3,4}{
        \draw[Sequiv] (\x,-1) -- (\x,-2);
    }
    
    \foreach \x in {0,1,2,3,4}{
        \foreach \y in {-1,-2}{
            \fill (\x,\y) circle[radius=0.07];
        }
    }
    \fill (0, 0) circle[radius=0.07];
    \fill (1, 0) circle[radius=0.07];
    \fill[red] (2, 0) circle[radius=0.07];
    \fill[red] (3, 0) circle[radius=0.07];
    \fill (4, 0) circle[radius=0.07];
\end{tikzpicture}
\end{equation*}

The subgroup decompositions arising from the maps in \labelcref{eq:subgroup-decomposition-spaces} correspond to the points $\E\orbitS(G)$ and $\E\orbitA(G)$, marked in red. The equivalences in the first row show that we have a subgroup decomposition for all of the collections in the table. More interestingly, the $S$-equivalence $\E \orbitS(G) \to \SpG$ induces a normaliser decomposition
\begin{equation*}
\begin{tikzcd}
    \lim_{\orbitS(G)^\op} \StMod(G/P) & \arrow[l, "\sim"'] \StMod(\SpG) \arrow[r, "\sim"] & \lim_{\sigma \in \osS(G)} \StMod(G/N_G(\sigma)) \\
    & \StMod(G/G) \arrow[sloped, ul, "\sim"'] \arrow[u] \arrow[ur] &
\end{tikzcd}
\end{equation*}
and hence (via the equivalences in the second row of the table) a normaliser decomposition for all the collections in the table. Finally, the same argument applied to the $S$-equivalence $\E \fusionS(G) \to \SpG$ gives centraliser decompositions for the collections $\SpG$, $\ApG$, and $\ZpG$.
\end{proof}

\begin{remark}
We don't have a condition analogous to Dwyer's \enquote{ampleness} for decompositions of the stable module ∞-category. Recall that $\sC$ is \emph{ample} if the natural map \[\sC_{\h G} \to (*)_{\h G} \simeq \B G\] induces an isomorphism on mod $p$ homology, and that there are subgroup, centraliser, and normaliser decompositions associated with $\sC$ if and only if $\sC$ is ample. This latter statement is due to the maps
\begin{equation*}
\begin{tikzcd}
    & \sd(\sC) \arrow[d] & \\
\E\orbitC(G) \arrow[r] & \sC & \arrow[l] \E\fusionC(G) 
\end{tikzcd}
\end{equation*}
being $\h G$-homotopy equivalences, \emph{i.e.} $G$-maps that are homotopy equivalences.
Since these maps are not always $S$-equivalences, there is no analogous condition for the stable module ∞-category. The same phenomenon is discussed in \cite[Remark~3.10]{dwyer98} in the context of the behaviour of spectral sequences associated with the decompositions.
\end{remark}

\appendix
\sectionspace
\section{All relevant natural transformations are counits}
\sectionspace
\label{sec:its-a-counit}

In this section, we prove the assertion made in \cref{rem:outline-of-factorisation-proof} that the natural transformation in each step of the outline of \cref{prop:condition-for-factoring-through-P(A)} is given by the counit of some adjunction. The only step for which this is not obvious is the following:
\begin{lemma}
Under the equivalence
\[
(i^*)^* (i_!)^* (\yoB)_! \simra (\yoB)_! \jshr \jpb
\]
given by \cref{prop:composing-lans,prop:second-Kan-lemma}, the natural transformation
\begin{equation}
\label{eq:top-nat-transf}
(i^*)^* (i_!)^* (\yoB)_! \to (\yoB)_!
\end{equation}
induced by the counit of the $(i^*)^* \dashv (i_!)^*$ adjunction corresponds to the natural transformation
\begin{equation}
\label{eq:bottom-nat-transf}
(\yoB)_! \jshr \jpb \to (\yoB)_!
\end{equation}
induced by the counit of the $\jshr \dashv \jpb$ adjunction.
\end{lemma}

\begin{proof}
This amounts to a large diagram chase; the diagram is reproduced below. We first explain why the diagram proves the lemma, then explain why the diagram commutes. Every map in the diagram is an equivalence. For functors $F, F' \in \Fun(\cC, \cD)$, we denote the mapping space of natural transformations from $F$ to $F'$ by $\Map(F, F')$.

The natural transformation \labelcref{eq:top-nat-transf} is an element of the top-left mapping space in the diagram, and the two vertical maps on the left-hand edge are induced by the equivalences of \cref{prop:composing-lans,prop:second-Kan-lemma}. We therefore aim to show that \labelcref{eq:top-nat-transf} is sent to \labelcref{eq:bottom-nat-transf} in the bottom-left corner of the region labelled $\circled{3}$.

Since \labelcref{eq:top-nat-transf} is induced by the counit of the $(i^*)^* \dashv (i_!)^*$ adjunction, it is sent to the identity natural transformation in the top-right corner of the square labelled $\circled{2}$. The maps along the right-hand edge of the diagram are either induced by cancelling inverse equivalences or by applying the identity $\yoA^*(i_!)^* \simeq j^* \yoB^*$ of \cref{diag:pullback-functors}. One can check that these maps send the identity natural transformation in the top-right corner of $\circled{2}$ to the identity natural transformation in the bottom-right corner of the diagram, which in turn is sent to \labelcref{eq:bottom-nat-transf} in the bottom-left corner of the region labelled $\circled{3}$.

Therefore, it remains to establish that the diagram commutes. The square labelled $\circled{2}$ commutes by definition of its left-hand edge; see \cref{prop:composing-lans}. The region labelled $\circled{3}$ also commutes by definition of its left-hand edge; see \cref{prop:second-Kan-lemma}. All the other regions are easily seen to commute.
\end{proof}
\begin{equation*}
\scalebox{0.85}{
    \begin{tikzcd}[ampersand replacement=\&,column sep=1.5cm]
% first LHS row:
        \Map((i^*)^* (i_!)^* (\yoB)_!, (\yoB)_!) \arrow[r, "(i^*)^* \dashv\, (i_!)^*"] \arrow[ddd] \arrow[rddd, phantom, "\circled{1}"] \& \Map((i_!)^* (\yoB)_!, (i_!)^* (\yoB)_!) \arrow[r, "\yoA^*"] \arrow[ddd] \arrow[rddd, phantom, "\circled{2}"] \& \Map(\yoA^* (i_!)^* (\yoB)_!, \yoA^* (i_!)^* (\yoB)_!) \arrow[d] \\
        \& \& \Map(\jpb \yoB^* (\yoB)_!, \yoA^* (i_!)^* (\yoB)_!) \arrow[d] \\
        \&  \& \Map(\jpb, \yoA^* (i_!)^* (\yoB)_!) \arrow[d] \\
% second LHS row:
        \Map((i^*)^* (\yoA)_! \jpb, (\yoB)_!) \arrow[rrdd, phantom, "\circled{3}"] \arrow[r, "(i^*)^* \dashv\, (i_!)^*"] \arrow[dd] \&  \Map((\yoA)_! \jpb, (i_!)^* (\yoB)_!) \arrow[r, "\yoA^*"] \arrow[dr, "(\yoA)_! \,\dashv\, \yoA^*"'] \& \Map(\yoA^* (\yoA)_! \jpb, \yoA^* (i_!)^* (\yoB)_!) \arrow[d] \\
        \& \& \Map(\jpb, \yoA^* (i_!)^* (\yoB)_!)  \arrow[d] \\
        \Map((\yoB)_! \jshr \jpb, (\yoB)_!) \arrow[r, "(\yoB)_! \,\dashv\, \yoB^*"] \&  \Map(\jshr \jpb, \yoB^* (\yoB)_!) \arrow[r, "\jshr \,\dashv\, \jpb"] \arrow[d] \& \Map(\jpb, \jpb \yoB^* (\yoB)_!)  \arrow[d] \\
        \& \Map(\jshr \jpb, \id) \arrow[r, "\jshr \,\dashv\, \jpb"'] \& \Map(\jpb, \jpb) \\
    \end{tikzcd}
}
\end{equation*}

\sloppy
\printbibliography
\fussy

\end{document}